\documentclass[11pt]{article}

\usepackage{amsmath, amssymb, amsthm, verbatim,enumerate,bbm,color} 
\numberwithin{equation}{section}

\usepackage{mathrsfs}

\usepackage[backref,colorlinks,citecolor=blue,bookmarks=true]{hyperref}

\usepackage{tikz}

\title{A Ramsey variant of the Brown-Erd\H{o}s-S\'os conjecture}

\author{Asaf Shapira \thanks{
School of Mathematics, Tel Aviv University, Tel Aviv 69978, Israel.
Email: asafico$@$tau.ac.il. Supported in part by ISF Grant 1028/16, ERC Consolidator Grant 863438 and NSF-BSF Grant 20196.}
\and
Mykhaylo Tyomkyn
	\thanks{Department of Applied Mathematics, Charles University, 11800 Prague, Czech Republic. Email: tyomkyn$@$kam.mff.cuni.cz. Supported in part by ERC Starting Grant 676632, ERC Synergy grant DYNASNET 810115 and the H2020-MSCA-RISE project CoSP- GA No. 823748.
}}

\date{\today}
\allowdisplaybreaks

\theoremstyle{plain}
\newtheorem{theorem}{Theorem}[section]
\newtheorem{lemma}[theorem]{Lemma}

\newtheorem{proposition}[theorem]{Proposition}

\newtheorem{problem}[theorem]{Problem}
\newtheorem{remark}[theorem]{Remark}
\newtheorem{definition}[theorem]{Definition}

\makeatletter
\def\moverlay{\mathpalette\mov@rlay}
\def\mov@rlay#1#2{\leavevmode\vtop{%
   \baselineskip\z@skip \lineskiplimit-\maxdimen
   \ialign{\hfil$\m@th#1##$\hfil\cr#2\crcr}}}
\newcommand{\charfusion}[3][\mathord]{
    #1{\ifx#1\mathop\vphantom{#2}\fi
        \mathpalette\mov@rlay{#2\cr#3}
      }
    \ifx#1\mathop\expandafter\displaylimits\fi}
\makeatother

\newcommand{\cupdot}{\charfusion[\mathbin]{\cup}{\cdot}}
\newcommand{\bigcupdot}{\charfusion[\mathop]{\bigcup}{\cdot}}

\makeatletter
\renewenvironment{proof}[1][\proofname]
{\par\pushQED{\qed}
	\normalfont\topsep6\p@\@plus6\p@\relax\trivlist
	\item[\hskip\labelsep\bfseries#1\@addpunct{.}]
	\ignorespaces}
{\popQED\endtrivlist\@endpefalse}
\makeatother

\newcommand{\C}{\mathcal C}
\newcommand{\E}{\mathcal E}
\newcommand{\F}{\mathcal F}

\newcommand{\G}{\mathcal G}

\newcommand{\T}{\mathcal T}

\newcommand{\eps}{\varepsilon}

\RequirePackage{color}\definecolor{RED}{rgb}{1,0,0}\definecolor{BLUE}{rgb}{0,0,1} 

\begin{document}
\date{}
\maketitle

\begin{abstract}

An $r$-uniform hypergraph ($r$-graph for short) is called linear if every pair of vertices belong to at
most one edge. A linear $r$-graph is complete if every pair of vertices are in exactly one edge. The famous
Brown-Erd\H{o}s-S\'os conjecture states that for every fixed $k$ and $r$, every linear $r$-graph with $\Omega(n^2)$ edges
contains $k$ edges spanned by at most $(r-2)k+3$ vertices.
As an intermediate step towards this conjecture, Conlon and Nenadov recently suggested to prove its natural Ramsey relaxation. Namely, that
for every fixed $k$, $r$ and $c$, in every $c$-colouring of a complete linear $r$-graph, one
can find $k$ monochromatic edges spanned by at most $(r-2)k+3$ vertices. We prove that this Ramsey version of the conjecture holds under the additional assumption that $r \geq r_0(c)$, and we show that for $c=2$ it holds for all $r\geq 4$.

\end{abstract}

\section{Introduction}\label{sec:intro}

The first result in extremal graph theory is probably Mantel's theorem stating that an $n$ vertex graph with more than
$n^2/4$ edges contains $3$ edges spanned by $3$ vertices, that is, a triangle.
This is of course just a special case of Tur\'an's theorem, one of the fundamental theorems in graph theory.
Tur\'an's theorem spurred an entire branch within graph theory of what is now called Tur\'an-type problems in graphs and hypergraphs \cite{Keevash},
as well as in other settings such as matrices and ordered graphs, see \cite{Tardos}.

One of the most notorious Tur\'an-type problems is a conjecture raised in the
early 70's by Brown, Erd\H{o}s and S\'os \cite{BES1,BES}. To state it we need a few definitions.
An $r$-uniform hypergraph ($r$-graph for short) ${\cal G}=(V,E)$ is composed of a vertex set $V$ and an edge set $E$ where every edge in $E$
contains precisely $r$ distinct vertices. An $r$-graph is \emph{linear} if every pair of vertices belong to at most one edge.
We call a set of $k$ edges spanned by at most $v$ vertices a $(v,k)$-configuration. Then the Brown--Erd\H{o}s--S\'os conjecture (BESC for short)
states that for every $k,r\geq 3$ and $\delta>0$ if $n \geq n_0 (k,r,\delta)$ then every linear $r$-graph on $n$ vertices with at least $\delta n^2$ edges
contains an $((r-2)k+3,k)$-configuration.

The simplest case of the BESC is when $r=k=3$. This special case was famously solved by
Ruzsa and Szemer\'edi \cite{RSz} and became known as the $(6,3)$-theorem.
To get a perspective on the importance of this theorem suffice it to say that the famous {\em triangle removal lemma} (see \cite{CFS2} for a survey) was devised in order to prove the $(6,3)$-theorem, that one of the first applications of Szemer\'edi's
regularity lemma \cite{Sz} was in \cite{RSz}, and that the $(6,3)$-theorem implies Roth's theorem \cite{Roth} on 3-term arithmetic progressions in dense sets of integers. Despite much effort the problem is wide open already for the next configuration, namely $(7,4)$.
As an indication of the difficulty of this case let us mention that it implies the notoriously difficult Szemer\'edi theorem \cite{SzThm0,SzThm} for $4$-term arithmetic progressions (see~\cite{Erd75}). Let us conclude this discussion by mentioning that the best result towards
the BESC was obtained 15 years ago by S\'ark\"ozy and Selkow \cite{Sarkozy_Selkow} who proved that
$f_3(n,k+2+\lfloor\log_2 k\rfloor,k)=o(n^2)$.\footnote{Here, $f_3(n,v,k)$ is the corresponding extremal number, i.e. the smallest $m$ such that every $3$-graph with $n$ vertices and $m$ edges contains a  $(v,k)$-configuration.} Since then, the only advancement was obtained by Solymosi and Solymosi \cite{SolymosiSolymosi} who improved the $f_3(n,15,10)=o(n^2)$ bound of \cite{Sarkozy_Selkow} to $f_3(n,14,10)=o(n^2)$. Conlon, Gishboliner, Levanzov and Shapira \cite{CGLS} have recently announced an improvement of the result of \cite{Sarkozy_Selkow} that replaces the $\log k$ term with $\log k/\log\log k$.

Given the difficulty of the BESC, researchers have recently looked at various relaxations of it.
For example, instead of looking at arbitrary $r$-graphs, one can look at those arising from a group, see \cite{Long,NST,Solym74,SolymosiWong,Wong}.
We will consider in this paper another relaxation of the BESC which was recently suggested independently by Conlon and Nenadov (private communications).
We say that a linear $r$-graph in {\em complete\footnote{Such an object is sometimes called an $r$-Steiner System (when $r=3$ this is a {\em Steiner Triple System}). Note that there are many non-isomorphic complete linear $r$-graphs on $n$ vertices.}} if every pair of vertices belong to exactly one edge.

\begin{problem}[Conlon, Nenadov]\label{p:BESR}
Prove that the following holds for every $r\geq 3$, $k\geq 3$, $c\geq 2$ and large enough $n \geq n_0(c,r,k)$: If ${\cal G}$ is an $n$-vertex
complete linear $r$-graph then in every $c$-colouring of its edges one can find $k$ edges of the same colour, which are spanned by at most $(r-2)k+3$ vertices.
\end{problem}

As we mentioned above, the BESC is a Tur\'an-type question, stating that enough edges force the appearance of certain configurations.
With this perspective in mind, Problem \ref{p:BESR} is its natural Ramsey weakening. Indeed, BESC implies its statement, as it gives the required monochromatic configuration in the most popular colour. The relation is analogous to the one between Szemer\'edi's theorem \cite{SzThm}
and Van der Waerden's theorem \cite{vdw}. In order to get a better feeling of this problem, we encourage the reader to convince themself of the folklore observation that Problem \ref{p:BESR} holds for $c=1$. A simple application of Ramsey's theorem also shows that Problem \ref{p:BESR} holds when $k=3$.

Our main result in this paper gives a positive answer to Problem \ref{p:BESR} assuming $r$ is large enough.
More precisely, we have the following.

\begin{theorem}\label{thm:maingeneral}
For every integer $c$ there exists $r_0=r_0(c)$ such that for every $r\geq r_0$ and integer $k\geq 3$ there exists $n_0=n_0(c,r,k)$ such that every $c$-colouring of a complete linear $r$-graph on $n>n_0$ vertices contains a monochromatic $((r-2)k+3,k)$-configuration.
\end{theorem} 

Note that even under assumptions of large uniformity it is unlikely that $((r-2)k+3,k)$ can ever be improved to $((r-2)k+2,k)$. Indeed, a conjecture by F\"uredi and Ruszink\'o~\cite{FR} states that for each $r\geq 3$ there exist arbitrarily large $r$-Steiner Systems without an $((r-2)k+2,k)$-configuration. That would preclude an extension of Theorem~\ref{thm:maingeneral} to $((r-2)k+2,k)$ even for $c=1$. The case $r=3$ of the F\"uredi-Ruszink\'o conjecture is an old conjecture by Erd\H{o}s~\cite{Erd73}, which was recently proved asymptotically, independently in~\cite{BW} and~\cite{GKLO}. 

In the important special case of $c=2$ we show that $r_0(2)$ can be chosen as small as $4$.
\begin{theorem}\label{thm:main}
For any integers $r\geq 4$ and $k\geq 3$ there exists $n_0=n_0(r,k)$ such that every $2$-colouring of a complete linear $r$-graph on $n>n_0$ vertices contains a monochromatic $((r-2)k+3,k)$-configuration.
\end{theorem}

While we believe that with some effort it should be possible to show that $r_0(2)=3$, it appears that completely removing
the assumption that $r$ is large enough as a function of $c$ would require a different approach. In particular, while
the case $k=3$ is an easy application of Ramsey's theorem, we do not know how to resolve Problem \ref{p:BESR} already
for $(c,r,k)=(3,3,4)$.

\subsection{Proof and paper overview}

The proof of Theorems \ref{thm:maingeneral} and \ref{thm:main} has two key ideas.
The first is to work with an auxiliary graph $B$ of ``bowties''. Every vertex $v$ in this graph
corresponds to a pair of intersecting\footnote{Since we only consider linear $r$-graphs, if two edges intersect, they intersect at a single vertex.
We will frequently use this fact throughout the paper.} edges of the $r$-graph ${\cal G}$. The graph $B$ contains edges
only between a vertex $b_1$, representing two intersecting edges $\{S_1,T\}$ of ${\cal G}$ and another vertex $b_2$,
representing two intersecting edges $\{S_2,T\}$ and only if the edges $S_1,S_2,T$ form a $(3r-3,r)$-configuration.
In Section \ref{sec:prelim} we will collect several preliminary observations regarding the graph $B$ and about edge-colourings of complete graphs. In Section~\ref{sec:mainproof}
we will prove our main results assuming $B$ has certain nice properties. This will reduce the proof to proving
Lemma \ref{lem:mainweak} which is the main technical part of the paper and is proved in Section \ref{sec:mainlemma}.
The second main idea of this paper is to define a somewhat subtle induction which will be used in order
to gradually ``grow'' $((r-2)k+3,k)$-configurations, for $k=3,4,\ldots$, and thus prove Lemma \ref{lem:mainweak}.
See Section \ref{sec:mainlemma} for an overview of this proof.

Perhaps one take-home message of this paper is that even when considering the Ramsey relaxation of the BESC (stated in Problem \ref{p:BESR}), and even after adding the assumption that $r \geq r_0(c)$, one still has to work quite hard in order to find the $((r-2)k+3,k)$-configurations of the BESC.

\paragraph*{Note added}
In the period when this paper was under review, Keevash and Long~\cite{KL} proved a density version of Theorem~\ref{thm:maingeneral} by applying (among other things) the notion of bow tie graphs which we introduce in this paper. 

\section{Preliminaries}\label{sec:prelim}

\subsection{Notation}
We use \emph{graph} in the standard meaning, i.e. referring to simple and undirected $2$-uniform graphs $G=(V,E)$, where $V$ is the set of \emph{vertices}, and $E\subseteq \binom{V}{2}$ are the \emph{edges} of $G$. We write $e(G)$ for $|E(G)|$.
We use the shorthand \emph{components} for connected components of a graph. For a vertex set $A\subseteq V$ we write $G[A]$ to denote the subgraph of $G$ induced on $A$, and similarly for disjoint vertex sets $A_1,\dots,A_m \subseteq V$ we use notation $G[A_1,\dots,A_m]$ for the induced multipartite subgraph between these sets.

We use \emph{$r$-graph} for $r$-uniform hypergraphs, denoted by script letters, i.e. $\G=(V,E)$, where $E\subseteq \binom{V}{r}$. We refer to $E$ as the set of \emph{hyperedges}. A hypergraph is \emph{linear} if no two hyperedges intersect in more than one vertex.
A \emph{complete linear $r$-graph} (also known as $r$-Steiner System) is a linear hypergraph corresponding to an edge-decomposition of a complete graph $K_n$ into copies of $K_r$. For a linear hypergraph $\G=(V,E(\G))$ the \emph{underlying graph} is the graph $G=(V,E(G))$, where $E(G)=\{e\in \binom{V}{2}: \exists e^+\in E(\G): e\subset e^+\}$. For instance, the underlying graph of a complete linear $r$-graph is always the complete graph.

For integers $v$ and $k$ we define a \emph{$(v,k)$-configuration} to be a hypergraph on $k$ hyperedges spanned by at most $v$ vertices.

We write $d_G(v)$, $d_{\G}(v)$ for degree of the vertex $v$ in the graph $G$ or hypergraph $\G$, respectively. Similarly $d_{avg}(G), d_{avg}(\G)$ denote the average degree in a graph $G$ or hypergraph $\G$. For a graph $G$ denote by $\T(G)$ the $3$-graph of triangles in $G$, and by $T(G)$ the number of triangles in $G$, that is $T(G)=|E(\T(G))|$.

We conclude by observing that for every $r \geq 2$ there is a unique (up to isomorphism) linear $r$-graph consisting of
$3$ edges on at most $3r-3$ vertices. For example, when $r=2$ this is a triangle, and the case $r=4$ is depicted in Figure 1.
For every $r$, we will use ${\cal C}^r_3$ to denote this unique configuration.

\subsection{The auxiliary graph $B$}

For $r\geq 3$, given a linear $r$-graph $\G$, define $B=B(\G)$, the \emph{bowtie graph} of $\G$, to be the following auxiliary graph. The vertices of $B$ will be \emph{bowties} of $\G$; that is, each vertex $b\in V(B)$ is a set of the form $b=\{S,T\}$ where $S,T\in E(\G)$ with $S\cap T=\{u\}$ for some vertex $u\in V(\G)$. We say that $u$ is the \emph{centre} of the bowtie $b$.
The edge set of $B$ is defined by
\begin{equation}\label{defB}
E(B)=\{b_1,b_2\in B: b_1=\{S_1,T\}, b_2=\{S_2,T\}, |S_1\cap S_2|=1, |S_1\cap S_2\cap T|=0 \}\;,
\end{equation}
see Figure 1 for an illustration.

We shall now state some basic properties of the graph $B$.

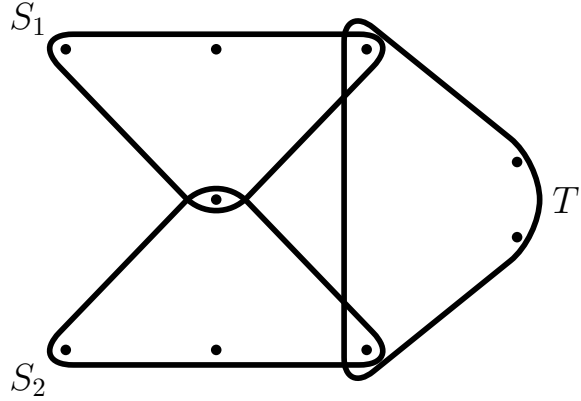
\begin{figure}\label{figure:63config}
\centering

	\begin{tikzpicture}
	\coordinate (v1) at (0,0);
	\coordinate (v2) at (0,2);
	\coordinate (v3) at (2,2);
	\coordinate (v4) at (-2,2);
	\coordinate (v5) at (0,-2);
	\coordinate (v6) at (2,-2);
	\coordinate (v7) at (-2,-2);
	\coordinate (v8) at (4,0.5);
	\coordinate (v9) at (4,-0.5);
	
	\foreach \i in {1,2,3,4,5,6,7,8,9}
	{
		\draw (v\i) node[fill=black,circle,minimum size=4pt,inner sep=0pt] {};
	}
	
	\draw [rounded corners = 6mm, line width = 0.75mm] (0,-0.4) -- (2.5,2.2) -- (-2.5,2.2) --cycle;
	\draw [rounded corners = 6mm, line width = 0.75mm] (0,0.4) -- (2.5,-2.2) -- (-2.5,-2.2) --cycle;
	\draw [rounded corners = 5mm, line width = 0.75mm] (1.7,2.6) -- (4.3,0.5) -- (4.3,-0.5) -- (1.7,-2.6) -- cycle;
	
	\draw (-2.5,2.4) node {\Large$S_1$};
	\draw (-2.5,-2.4) node {\Large$S_2$};
	\draw (4.65,0) node {\Large$T$};
	\end{tikzpicture}
\caption{An illustration of ${\cal C}^4_3$, the unique $(9,3)$-configuration in a linear $4$-graph. In $B$, such a configuration is represented by a triangle on the vertices $b_1=\{S_1,T\}$, $b_2=\{S_2,T\}$ and $b_3=\{S_1,S_2\}$.}
\end{figure}

\begin{proposition}\label{prop:B}
Suppose that $\G=(V,E(\G))$ is a linear $r$-graph, $G$ is the underlying graph, and $B=B(\G)$. Then the following statements hold.
\begin{enumerate}
\item[(1)] For any pair of bowties $b_1=\{S_1,T\}$ and $b_2=\{S_2,T\}$ with $\{b_1,b_2\}\in E(B)$ we have that $\{S_1,S_2,T\}$ is a
copy of ${\cal C}^r_3$, and their centres, given by $S_1\cap T=:\{u_2\}$, $S_2\cap T=:\{u_2\}$, and $S_1 \cap S_2=:\{u_3\}$ form a triangle in $G$.
\item[(2)] Every triangle $\{u,v,w\}\subset V(G)$ which is not contained in a hyperedge of $\G$ uniquely defines a copy of ${\cal C}^r_3$ composed of the three edges $\{Q,S,T\}\subset E(\G)$, where $\{u,v\}\subset Q, \{u,w\}\subset S$, and $\{v,w\}\subset T$. The bowties $\{Q,S\},\{Q,T\}$ and $\{S,T\}$ form a triangle in $B$.
\item[(3)] All degrees in $B$ are even, and $\Delta(B)\leq 2(r-1)^2$.
\item[(4)] For every vertex $u\in V(\G)$ the set $B_u=\{b\in B: u \text{ is a centre of } b\}$ is independent in $B$.
\item[(5)] With the above notation, for every $u\in V(\G)$,
$$
\sum_{b\in B_u}d_B(b)=2(d_{\T(G)}(u)-\binom{r-1}{2}d_\G(u))\;.
$$
\end{enumerate}
\end{proposition}
\begin{proof}
If $b_1$ and $b_2$ are as in (1), then, by inclusion-exclusion,
$$
|S_1\cup S_2\cup T|=r+r+r-1-1-1+0=3r-3\;,
$$
implying that $S_1,S_2,T$ form a copy of ${\cal C}^r_3$.
Moreover, $\{u_1,u_2,u_3\}$ form a triangle in $G$, as $\{u_1,u_2\}\subset T\in E(\G)$, $\{u_1,u_3\}\subset S_1\in E(\G)$, and $\{u_2,u_3\}\subset S_2\in E(\G)$. This shows (1).

Property (2) is similar. If $u,v,w$ is a triangle, then, by definition of $G$, every pair of these vertices belong to an edge of ${\cal G}$. The
assumption that the triple $\{u,v,w\}$ does not belong to an edge of ${\cal G}$ implies that there are three distinct edges $Q,S,T$ each containing
two of these vertices. The assumption that ${\cal G}$ is linear and the fact that each two of these edges intersect in exactly one of the vertices $\{u,v,w\}$, implies that they have no vertex in common.
As in the previous paragraph this means that $|Q \cup S \cup T|=3r-3$ so $Q,S,T$ form a copy of ${\cal C}^r_3$. These observations
and the definition of $B$ in (\ref{defB}) also guarantee that the bowties $\{Q,S\},\{Q,T\}$ and $\{S,T\}$ form a triangle in $B$.
	
For a bowtie $b=\{S,T\}\in B$ there are $(r-1)^2$ vertex pairs in $S\cup T$  not contained in $S$ or in $T$. Hence, $b$ is contained in at most $(r-1)^2$ possible copies of ${\cal C}^r_3$ on edges $\{Q,S,T\}$, each of which would give rise to two neighbours of $b$ in $B$: the bowties $\{Q,S\}$ and $\{Q,T\}$. Hence, $d_B(b)$ is even and at most $2(r-1)^2$, establishing (3).
	
Property (4) is follows from (1) and the fact that every three bowties from $B_u$ have $u$ as the unique common vertex, thus cannot form a ${\cal C}^r_3$.
	
As for property (5), by (1) and (2), the sum $\sum_{b\in B_u}d_B(b)$ equals twice the number of triangles in $G$ that contain $u$ and are not contained in a hyperedge of $\G$. Observing that the remaining triangles containing $u$ are partitioned into sets of size $\binom{r-1}{2}$ by the hyperedges of $\G$ yields the desired statement.
\end{proof}

\begin{remark}
Note that item (1) above implies that the edges of $B$ have a ``natural'' triangle decomposition in which the edge connecting the vertices $b_1=\{S_1,T\}$ and $b_2=\{S_1,T\}$ is put in the triangle spanned by the vertices $b_1,b_2$ and $b_3=\{S_1,S_2\}$.
This means that an equivalent way to define $B$ is to put, for every
${\cal C}^r_3$ consisting of edges $S_1,S_2,T$, a triangle on the vertices $b_1,b_2,b_3$. It is also worth noting
that these triangles are not the only triangles in $B$, that is, not all triangles in $B$ correspond to
a copy of ${\cal C}^r_3$ in $B$. For example, when $r=3$ one can take a ${\cal C}^3_3$ consisting of edges $A,B,C,$ and then add another edge $D$ on the vertices of degree $1$ (this is known as the Pasch-configuration). Then every three of these four edges forms a ${\cal C}^r_3$ hence
the three bowties $\{A,D\}$, $\{B,D\}$, $\{C,D\}$ form a triangle in $B$ although they do not form a ${\cal C}^r_3$.
\end{remark}

The following lemma establishes a connection between large connected components in $B$ and $((r-2)k+3,k)$-configurations in $\G$. The below constant of $r^{10k^2}$ is rather generous (it can be easily made polynomial in $k$), but chosen so to streamline the proof. This has no overall impact on the strength of Theorems~\ref{thm:maingeneral} and~\ref{thm:main}.
\begin{lemma}\label{lem:smallcomps}
If $B=B(\G)$ has a component of size at least $r^{10k^2}$, then $\G$ contains an $((r-2)k+3,k)$-configuration.
\end{lemma}
\begin{proof}
Suppose that $B$ has a component of size at least $r^{10k^2}$. Since by Proposition~\ref{prop:B}(3), $\Delta(B)\leq 2(r-1)^2$, it follows that $B$ contains a path of length $k^2$ (as $\Delta(B)^{ {k^2}}<r^{10k^2}$). Let $b_1b_2\dots b_p$ be the longest path in $B$, where the vertices are numbered along the path; by the above we can assume that $p\geq k^2$.

For each $1\leq i\leq p$, consider the  hypergraph $\F_i$ of all hyperedges belonging to one of the bowties $b_1,\dots,b_i$. Since $\F_1=b_1$, we have that $\F_1$ is a  $(2(r-2)+3,2)$-configuration. Suppose now that for some $i<p$, the hypergraph $\F_i$ is an $((r-2)m_i+3,m_i)$-configuration, where $m_i=e(\F_i)$, and consider $\F_{i+1}$. Since $b_{i+1}$ and $b_i$ are adjacent in $B$, by Proposition~\ref{prop:B}(1), we can write $b_i=\{Q,S\}$ and
$b_{i+1}=\{S,T\}$, where $\{Q,S,T\}$ is a ${\cal C}^r_3$ in $\G$, and note that $Q, S\in E(\F_i)$. Regarding $T$ we have the following two options:
	
\textbf{Case 1}: $T\in E(\F_{i})$. Then $\F_{i+1}=\F_i$, $m_{i+1}=m_i$, and $\F_{i+1}$ is an $((r-2)m_{i+1}+3,m_{i+1})$-configuration.
	
\textbf{Case 2:} $T \notin E(\F_i)$. Then $m_{i+1}=m_i+1$, and, since $|T\setminus V(\F_i)|\leq |T\setminus (Q\cup S)|=r-2$, we have
$$|V(\F_{i+1})|\leq (r-2)m_i+3+(r-2)=(r-2)(m_i+1)+3,$$
so $\F_{i+1}$ is an $((r-2)m_{i+1}+3,m_{i+1})$-configuration.
	
Furthermore, since $\{b_1,\dots b_p\}\subseteq V(B(\F_p))$, we have
$$k^2 \leq p\leq |B(\F_p)|\leq |E(\F_p)|^2,$$
implying $|E(\F_p)|\geq k$.
Since in each step $m_i$ increased by at most $1$, by the discrete intermediate value theorem, for some $i\leq p$ the hypergraph $\F_i\subseteq \G$ is an $((r-2)k+3,k)$-configuration.
\end{proof}

\subsection{Ramsey multiplicity}
We shall need a Ramsey multiplicity bound for $c$-edge colourings of the complete graph. Recall that $T(G)$ denotes the number of triangles in $G$.
\begin{lemma}\label{lem:ramseymult}
	For every $c\geq 3$ there exist $n_0=n_0(c) $ and $\alpha_c>0$, such that for all $n>n_0$ in every $c$-colouring of the complete graph: $E(K_n)=\bigcupdot_{i=1}^cG_i$
	there will be a colour class $G_i$ satisfying	
	$$T(G_i)\geq\alpha_c\sum_{u\in K_n}\binom{d_{G_i}(u)}{2} \ \text{ and } \ \sum_{u\in K_n}\binom{d_{G_i}(u)}{2}\geq \alpha_c n^3\;.
	$$
\end{lemma}	
\begin{proof}
	Suppose that for some integer $n>0$ we have an edge colouring $E(K_n)=\bigcupdot_{i=1}^cG_i$. By a standard fact from Ramsey theory, there exists a constant $1\geq \rho_c>0$ such that for large enough $n$ and some $1\leq i\leq c$ we will have
	\begin{equation}\label{eq:ramseymult}
	T(G_i)\geq \rho_c\binom{n}{3}.
	\end{equation}
	Write
	$$S_i:=\sum_{u\in K_n}\binom{d_{G_i}(u)}{2}, \ 
	$$
	and note that $S_i$ counts the `cherries' (copies of $K_{1,2}$) in $G_i$.
	Observe that, evaluating $S_i$ via cherries of colour $i$ in all possible vertex triples, we count $3$ for each triangle and at most $1$ for each of the remaining triples, resulting 
	in
	$$3T(G_i) \leq S_i\leq \binom{n}{3}+2T(G_i)\;.
	$$
	Therefore,
	$$\frac{T(G_i)}{S_i}\geq\frac{T(G_i)}{\binom{n}{3}+2T(G_i)}\geq \frac{\rho_c}{1+2\rho_c}=:3\alpha_c>\alpha_c\;.
	$$
	Here we used~\eqref{eq:ramseymult} and the fact that for all $a,x>0$ the function $\frac{x}{a+2x}$ is monotone increasing in $x$.
	Furthermore,
	$$S_i\geq 3T(G_i)\stackrel{\eqref{eq:ramseymult}}{\geq} 3\rho_c\binom{n}{3}>\frac{\rho_c}{3}n^3>\alpha_c n^3\;.
	$$
\end{proof}

\noindent
For $c=2$, we can apply Goodman's theorem~\cite{Gd} to prove a stronger version of the above statement. This will be needed in the proof of Theorem~\ref{thm:main}.
\begin{lemma}\label{lem:Goodman}
	For every $\eps>0$ there exists $n_0=n_0(\eps)$ such that for all $n>n_0$ in every $2$-colouring of the complete graph $K_n$
	there will be a colour class $G$ satisfying
	$$T(G)\geq (\frac{1}{6}-\eps)\sum_{u\in K_n}\binom{d_G(u)}{2} \ \text{ and } \ \sum_{u\in K_n}\binom{d_G(u)}{2}\geq \frac{n^3}{10^6}\;.
	$$
\end{lemma}
\begin{proof}
	We can assume that $1/100>\eps>0$, and suppose that for some integer $n>0$ we have an edge colouring $E(K_n)=G\cupdot H$.
	
	As in the standard proof of Goodman's theorem, and similarly to the proof of Lemma~\ref{lem:ramseymult}, we count the copies of $K_{1,2}$ in $G$ and in $H$ around each vertex, and take the sums, obtaining the quantity
	\begin{equation}\label{eq:Goodman1}
	S:=\sum_{u\in K_n} \binom{d_G(u)}{2}+\binom{d_H(u)}{2}.
	\end{equation}
	By convexity in~\eqref{eq:Goodman1}, we have
	\begin{equation}\label{eq:convex}
	S\geq 
	2n\binom{\frac{n-1}{2}}{2}.
	\end{equation}
	On the other hand, evaluating $S$ via monochromatic cherries in all possible vertex triples, we count $3$ for each monochromatic triangle and $1$ for each of the remaining triples, resulting in
	\begin{equation}\label{eq:Goodman2}
	S=\binom{n}{3}+2T(G)+2T(H).
	\end{equation}
	Combining~\eqref{eq:convex} and~\eqref{eq:Goodman2},
	for sufficiently large $n$ we obtain
	$
	T(G)+T(H)\geq (\frac{1}{4}-\eps)\binom{n}{3}
	$ -- this is Goodman's theorem in its approximate form.
	
	In particular, for large $n$ we will have
	$$\frac{T(G)+T(H)}{S}\stackrel{\eqref{eq:Goodman2}}{=}\frac{T(G)+T(H)}{\binom{n}{3}+2(T(G)+T(H))}\geq \frac{1/4-\eps}{1+2/4-2\eps}\geq\frac{1}{6}-\eps\;,
	$$
	where in the first inequality we used the fact that, for $a,x>0$, the function $\frac{x}{a+2x}$ is monotone increasing in $x$.
	It follows by~\eqref{eq:Goodman1} that for large $n$ we have
	$$T(G)+T(H)\geq (\frac{1}{6}-\eps)\left(\sum_{u\in K_n} \binom{d_G(u)}{2}+\sum_{u\in K_n} \binom{d_H(u)}{2}\right)\;.
	$$
	Using the fact that any $a,b,c,d>0$ satisfy $\frac{a+b}{c+d}\leq \max\{\frac{a}{c},\frac{b}{d}\}$, we deduce that one of the colour classes, say $G$, will satisfy
	\begin{equation}\label{eq:Gisgoodman}
	T(G)\geq (\frac{1}{6}-\eps)\sum_{u\in K_n}\binom{d_G(u)}{2}.
	\end{equation}
	Suppose now that $G$ fails to satisfy the second requirement, i.e. $\sum_{u\in K_n}\binom{d_G(u)}{2}< 10^{-6}n^3$.
	In that case, by convexity,
	$$n\binom{d_{avg}(G)}{2}\leq \sum_{u\in K_n}\binom{d_G(u)}{2}< 10^{-6}n^3\;.
	$$ So, $d_{avg}(G)<n/100$, which implies $d_{avg}(H)>98n/100$, and, applying convexity again, we obtain
	\begin{equation}\label{eq:sconvex}
	S\geq \sum_{u\in K_n}\binom{d_H(u)}{2}\geq n  \binom{d_{avg}(H)}{2}\geq n  \binom{98n/100}{2}\;.
	\end{equation}

	By~\eqref{eq:Goodman2},~\eqref{eq:sconvex}, and the fact that $T(G)\leq \sum_{u\in K_n}\binom{d_G(u)}{2}< 10^{-6}n^3$, 
	we have
	\begin{align*}
	T(H)
	&\stackrel{\eqref{eq:Goodman2},\eqref{eq:sconvex}}{\geq} \frac{1}{2}( n \binom{98n/100}{2}-\binom{n}{3}-2T(G))\\
	&>\frac{n}{2} \binom{98n/100}{2}-\frac{1}{2}\binom{n}{3}-10^{-6}n^3\\
	&>\frac{1}{6}n\cdot \binom{n}{2}\geq \frac{1}{6}\sum_{u\in K_n}\binom{d_H(u)}{2}\;,
	\end{align*}

	so in this case $H$ satisfies the requirements of the lemma, with room to spare.
\end{proof}
\section{Structure of the auxiliary graph}\label{sec:mainproof}
\subsection{Main proof of Theorems~\ref{thm:maingeneral} and~\ref{thm:main}}
In order to streamline the statements of the upcoming lemmas, we shall formulate two setups that correspond to the settings of Theorems~\ref{thm:maingeneral} and~\ref{thm:main}, respectively.

\medskip
\noindent
\textbf{Setup $A_c(r)$} [$c,r\geq 3$ integers]. Suppose that $n>\max\{r/\alpha_c^2, n_0(c)\}$, where $n_0$ and $\alpha_c$ are as defined in Lemma~\ref{lem:ramseymult}.  Suppose we have a set $U$ of $n$ vertices, and a $c$-coloured complete linear $r$-graph on $U$, with colour classes $\G_1,\dots,\G_c$. Denote by $G_1,\dots,G_c$ the underlying graphs of $\G_1,\dots,\G_c$, respectively. So, $\bigcupdot_{1\leq i\leq c} E(G_i)=\binom{U}{2}$, and the conditions of Lemma~\ref{lem:ramseymult} hold. Suppose that $G_i$ is the graph satisfying the assertion of Lemma~\ref{lem:ramseymult}. Write $G=G_i$ and $\G=\G_i$, and let $B=B(\G)$.

\medskip
\noindent
\textbf{Setup $A_2(\eps,r)$} [$\eps>0$; $r\geq 3$ integer]. Suppose that $n>\max\{10^6 r/\eps, n_0(\eps)\}$, where $n_0$ is as defined in Lemma~\ref{lem:Goodman}. Suppose we have a set $U$ of $n$ vertices, and a $2$-coloured complete linear $r$-graph on $U$, with colour classes $\G$ and $\mathcal{H}$. Denote by $G$ and $H$ the underlying graphs of $\G$ and $\mathcal{H}$. So, $E(G)\cupdot E(H)=\binom{U}{2}$, and the conditions of Lemma~\ref{lem:Goodman} hold. Suppose that $G$ satisfies the assertion of Lemma~\ref{lem:Goodman}, and let $B=B(\G)$.

\medskip
The next two lemmas give in the above setups lower bounds on the order and density of $B$. The proof of Lemma~\ref{lem:9minuseps} will be given in the next subsection.

\begin{lemma}\label{lem:basic}
	For any $c\geq 3$ and $r\geq 3$ in Setup $A_c(r)$ we have
	\begin{equation}\label{eq:sizeofB}
|B|=\sum_{u\in U}\binom{\frac{d_G(u)}{r-1}}{2}\geq \frac{\alpha_c}{2r^2}n^3.
\end{equation}
	For any $\eps>0$ and $r\geq 4$ in Setup $A_2(\eps,r)$ we have
	\begin{equation}\label{eq:sizeofB2}
|B|=\sum_{u\in U}\binom{\frac{d_G(u)}{r-1}}{2}\geq \frac{n^3}{2\cdot 10^{6}r^2}.
\end{equation}
\end{lemma}
\begin{proof}
	A vertex $u\in U$ with $d_\G(u)=m$ satisfies $d_G(u)=(r-1)m$, and is centre to $\binom{m}{2}$ bowties. Therefore, in Setup $A_c(r)$ we have \eqref{eq:sizeofB} as an immediate consequence of Lemma~\ref{lem:ramseymult}, and in Setup $A_2(\eps,r)$ we have \eqref{eq:sizeofB2} as a consequence of Lemma~\ref{lem:Goodman}.
\end{proof}
\begin{lemma}\label{lem:9minuseps}
	For every $c\geq 3$ and $r\geq 3$ in Setup $A_c(r)$ we have
	$$d_{avg}(B)\geq \alpha_cr^2\;.
	$$
	For every $\eps>0$, in Setup $A_2({\eps},r)$ we have
	$$d_{avg}(B)\geq (r-1)^2-7{r^2}\eps\;.$$
\end{lemma}
\noindent
Next, we study the connected components of $B$. To this end, the following definition will be of central importance.
\begin{definition}
For any linear $r$-graph $\G$ and the associated bowtie graph $B=B(\G)$ call a component of $B$ \emph{dense} if its average degree is at least $3(r-1)$.	
\end{definition}
The next two lemmas claim that in the appropriate setups $B$ will have either a big component, or a large number of dense components. Their proofs will be given in the next subsection.
\begin{lemma}\label{lem:densecompsgeneral}
	For any $c\geq 3$ there exists $r_0=r_0(c)$ such that for all $r\geq r_0$ the following holds in Setup $A_c(r)$. For every $k\geq 3$ there exists 
	$\beta=\beta(c,r,k)>0$ such that 
	either $B$ has a component of size at least $r^{10k^2}$, or
	$B$ has at least $\beta n^3$ dense components.
\end{lemma}	
\begin{lemma}\label{lem:manydensecomps}
	For any $r\geq 4$ and $k\geq 3$ there exists $\eps=\eps(r,k)>0$ such that in Setup $A_2(\eps,r)$ either $B$ has a component of size at least $r^{10k^2}$, or $B$ has at least $r^{-30k^2}n^3$ dense components.
\end{lemma}
When $B$ has many dense components, we can arrange some of them in a particularly helpful way.
\begin{lemma}\label{lem:overlapcomps}
	For every $r\geq 3$, $k\geq 3$, $\beta>0$
	and $n>2rk\cdot r^{10k^2}\beta^{-1}$ the following holds. Suppose that $|U|=n$ and $\G$ is a linear $r$-graph on $U$ such that its bowtie graph $B=B(\G)$ has at least $\beta n^3$ dense components, each of which has at most $r^{10k^2}$ vertices.
	
	Then there exist a vertex $u_0\in V(\G)$, a hyperedge $T_0\in E(\G)$ with $u_0\in T_0$, a set of $2rk$ further hyperedges $\T^0=\{T^0_1,\ldots ,T^0_{2rk}\}\subseteq E(\G)$, and a set of distinct dense components of $B$, $\C=\{C^1,\ldots, C^{2rk}\}$ such that, for each $\ell$, we have $u_0\in T_\ell^0$ and $\{T_0,T^0_\ell\}\in C^\ell$.
\end{lemma}

\begin{proof}
	Suppose that $r$, $k$, $n$ and $\G$ are as stated above. Denote by $B'$ the union of all components of $B$ which are dense and have at most $r^{10k^2}$ vertices. By assumption, and since each component has at least one vertex, we have
	$$|B'|\geq \beta n^3\;.$$
	For a vertex $u\in V(\G)$ denote by $B'_u$ the set off all bowties in $B'$ whose centre is $u$. Since each bowtie has a unique centre, by averaging, there will be a vertex $u_0\in U$ such that
	$$|B'_{u_0}|\geq \beta n^2\;.
	$$
	Since $u_0$ belongs to at most $n$ hyperedges, by further averaging, there will be a hyperedge $T_0$ with $u_0\in T_0$ such that
	$$|\{b\in B'_{u_0}:b=\{T_0,T\} \}|\geq \beta n\;.
	$$
	Moreover, since, by assumption, each component in $B'$ is of size at most $r^{10k^2}$, at least $r^{-10k^2}\beta n>2rk$ of the above bowties will belong to distinct components.
\end{proof}
Lastly, the following lemma, which will be proved in the next section, will in conjunction with the lemmas established in this section, readily prove Theorems~\ref{thm:maingeneral} and~\ref{thm:main}.
\begin{lemma}\label{lem:mainweak}
	Suppose that $r\geq 3$, $\G$ is a linear $r$-graph, and let $B=B(\G)$ be its bowtie graph. Suppose that $k\geq 3$, and that there exist a vertex $u_0\in V(\G)$, a hyperedge $T_0\in E(\G)$ with $u_0\in T_0$, a set of $2rk$ further hyperedges $\T^0=\{T^0_1,\ldots ,T^0_{2rk}\}\subseteq E(\G)$, and a set of distinct dense components of $B$, $\C=\{C^1,\ldots, C^{2rk}\}$ such that, for each $\ell$, we have $u_0\in T_\ell^0$ and $\{T_0,T^0_\ell\}\in C^\ell$. Then $\G$ contains an $((r-2)k+3,k)$-configuration.
\end{lemma}
\begin{proof}[Proof of Theorem~\ref{thm:maingeneral} and Theorem~\ref{thm:main}]
	For $c\geq 3$, let $r_0(c)$ be as in Lemma~\ref{lem:densecompsgeneral}. For $r\geq r_0$ and $k\geq 3$ let $\beta:=\beta(c,r,k)$ be as Lemma~\ref{lem:densecompsgeneral}, and 
	$$n_0(c,r,k):=\max\{r/\alpha_c^2,n_0(c),2rk\cdot r^{10k^2}\beta^{-1}\}\;,$$
	where $\alpha_c$ and $n_0(c)$ are as in Lemma~\ref{lem:ramseymult}. In particular, for all $n>n_0$, given a $c$-colouring of a complete linear $r$-graph on a set $U$ of $n$ vertices, we can assume Setup $A_c(r)$ and the conditions of Lemma~\ref{lem:densecompsgeneral}.
	
	For $c=2$, $r\geq 4$ and $k\geq 3$, let $\eps:=\eps(r,k)$ be as in Lemma~\ref{lem:manydensecomps}, let $\beta:=r^{-30k^2}$, and let
	$$n_0(r,k):=\max\{10^6r/\eps, n_0(\eps), 2rk\cdot r^{10k^2}\beta^{-1}\}\;,$$
	where $n_0(\eps)$ is as in Lemma~\ref{lem:Goodman}. In particular, we can assume Setup $A_2(\eps,r)$ for all $n>n_0$ and any $2$-colouring of a complete linear $r$-graph on a set $U$ of $n$ vertices.
	
	Then, by Lemma~\ref{lem:densecompsgeneral} or Lemma~\ref{lem:manydensecomps} respectively, either $B$ has a component of size at least $r^{10k^2}$, in which case $\G$ contains an $((r-2)k+3,k)$-configuration by Lemma~\ref{lem:smallcomps}, or $B$ has at least $\beta n^3$ dense components, and we can assume that each of them is of size at most $r^{10k^2}$.
	
	In that case, by Lemma~\ref{lem:overlapcomps}, $\G$ will satisfy the conditions of Lemma~\ref{lem:mainweak}. Invoking it, we conclude that $\G$ contains an $((r-2)k+3,k)$-configuration.
\end{proof}

%

\subsection{Proofs of the technical lemmas}

\begin{proof}[Proof of Lemma~\ref{lem:9minuseps}]
    By Proposition~\ref{prop:B}(5),
	and using that $d_G(u)=(r-1)d_{\G}(u)$,
	we have
	\begin{align}\label{eq:TRG}
	\sum_{b\in B}d_B(b)&=\sum_{u\in U}\sum_{b\in B_u}d_B(b)=\sum_{u\in U}2(d_{\T(G)}(u)-\binom{r-1}{2}d_\G(u))\nonumber\\
	&=6T(G)-\sum_{u\in U}(r-2)d_G(u)=6T(G)-2(r-2)e(G)\nonumber\\
	&\geq 6T(G)-rn^2.
	\end{align}
	By assumptions of Setup $A_c(r)$ and by Lemma~\ref{lem:ramseymult} we get
	$$6T(G)\geq 6\alpha_c\sum_{u\in U}\binom{d_G(u)}{2}, \ \text{ and } \ rn^2<{\alpha_c^2} n^3 \leq {\alpha_c}\sum_{u\in U}\binom{d_G(u)}{2}\;,
	$$
	implying that
	\begin{equation}\label{eq:TRGc}
	6T(G)-rn^2\geq 5{\alpha_c}\sum_{u\in U}\binom{d_G(u)}{2}\;.
	\end{equation}
	Hence,
\begin{align*}
	d_{avg}(B)&=\frac{\sum_{b\in B}d_B(b)}{|B|}\stackrel{\eqref{eq:TRG}, \eqref{eq:TRGc}}{\geq} \frac{5\alpha_c\sum_{u\in U}\binom{d_G(u)}{2}}{|B|}\nonumber \\
	&\geq \frac{5\alpha_c\cdot (r-1)^2\sum_{u\in U}\binom{\frac{d_G(u)}{r-1}}{2}}{|B|}\stackrel{\eqref{eq:sizeofB}}{=}\frac{5\alpha_c\cdot (r-1)^2|B|}{|B|}>\alpha_cr^2\;.
\end{align*}
	Similarly, by assumptions of Setup $A_2({\eps}, r)$ and by Lemma~\ref{lem:Goodman} we get
	$$6T(G)\geq (1-6{\eps})\sum_{u\in U}\binom{d_G(u)}{2}, \ \text{ and } \ rn^2< \frac{\eps n^3}{10^6} \leq  {\eps}\sum_{u\in U}\binom{d_G(u)}{2}\;,
	$$
	implying
	\begin{equation}\label{eq:TRG2}
	6T(G)-rn^2\geq (1-7\eps)\sum_{u\in U}\binom{d_G(u)}{2}\;.
	\end{equation}
	Therefore,
	$$
	d_{avg}(B)=\frac{\sum_{b\in B}d_B(b)}{|B|}\stackrel{\eqref{eq:TRG},\eqref{eq:TRG2}}{\geq} \frac{(1-7\eps)\sum_{u\in U}\binom{d_G(u)}{2}}{|B|}\stackrel{\eqref{eq:sizeofB2}}{\geq}\frac{(1-7\eps)\cdot (r-1)^2|B|}{|B|}>(r-1)^2-7 r^2\eps\;.
	$$
\end{proof}
\noindent
\begin{proof}[Proof of lemma~\ref{lem:densecompsgeneral}]
	Set $r_0:=\lceil 6/\alpha_c\rceil +1$, where $\alpha_c$ is as in Lemma~\ref{lem:ramseymult}. Then for $r\geq r_0$, in Setup $A_c(r)$, by Lemma~\ref{lem:9minuseps}, we obtain	
	\begin{align*}
	6r&<{\alpha_cr^2}\leq d_{avg}(B)=\frac{\sum_{b\in B}d_B(b)}{|B|}=\frac{1}{|B|}\sum_{C}|C|\frac{\sum_{b\in C}d_C(b)}{|C|}=\sum_{C}\frac{|C|}{|B|}d_{avg}(C)\\
	&=\sum_{C:d_{avg}(C)\geq 3r}\frac{|C|}{|B|}d_{avg}(C)+
	\sum_{C:d_{avg}(C)< 3r}\frac{|C|}{|B|}d_{avg}(C).
	\end{align*}
	Let $B_1$ be the union of components of average degree at least $3r$, and note that all of these components are dense; let $B_2=B\setminus B_1$. Then, since by Proposition~\ref{prop:B}(3), $\Delta(B)\leq 2(r-1)^2$,
	$$6r< \sum_{C\subseteq B_1}2(r-1)^2\frac{|C|}{|B|}+\sum_{C\subseteq B_2}3r\frac{|C|}{|B|}=\frac{2(r-1)^2|B_1|+3r(|B|-|B_1|)}{|B|}\leq\frac{3r^2|B_1|+3r|B|}{|B|}\;,
	$$
	which, by Lemma~\ref{lem:basic}, gives
	$$|B_1|>\frac{|B|}{r}{\geq} \frac{\alpha_cn^3}{2r^3}\;.
	$$
	If no component in $B$ has more than $r^{10k^2}$ vertices, then the number of dense components in $B$ must be at least
	$$\frac{|B_1|}{r^{10k^2}}\geq \frac{\alpha_c}{2r^3\cdot r^{10k^2}}n^3=:\beta n^3\;.
	$$
\end{proof}
\begin{proof}[Proof of Lemma~\ref{lem:manydensecomps}]
	Let $f(r,k)=r^{10k^2}$ and $\eps={1}/({14r^4f(r,k)})$. In Setup $A_2(\eps, r)$,
	applying Lemma~\ref{lem:9minuseps}, 
	and	averaging over the components $C\subseteq B$, we get
	\begin{align*}
	(r-1)^2-\frac{1}{2r^2f(r,k)}&\leq d_{avg}(B)= \frac{\sum_{b\in B}d_B(b)}{|B|}=\frac{1}{|B|}\sum_{C}|C|\frac{\sum_{b\in C}d_C(b)}{|C|}\\
	&=\sum_{C}\frac{|C|}{|B|}d_{avg}(C).
	\end{align*}
	Suppose now that all components of $B$ are of size at most $f(r,k)$.
	Then, each $d_{avg}(C)$ is a rational number with denominator bounded above by $f(r,k)$. Hence, we can write
	\begin{align*}
	\sum_{C}\frac{|C|}{|B|}d_{avg}(C)
	&=\sum_{C:d_{avg(C)}\geq (r-1)^2}\frac{|C|}{|B|}d_{avg}(C)+\sum_{C:d_{avg(C)}\leq (r-1)^2-1/f(r,k)}\frac{|C|}{|B|}d_{avg}(C).
	\end{align*}
	\noindent
	Let $B_1$ be the union of components of average degree at least $(r-1)^2$, and note that since $r\geq 4$ all of these components are dense; let $B_2=B\setminus B_1$. Since by Proposition~\ref{prop:B}(3), $\Delta(B)\leq 2(r-1)^2$, we obtain
	\begin{align*}
	(r-1)^2-\frac{1}{2r^2f(r,k)}&\leq \sum_{C\subseteq B_1}2(r-1)^2\frac{|C|}{|B|}+\sum_{C\subseteq B_2}((r-1)^2-\frac{1}{f(r,k)})\frac{|C|}{|B|}\\
	&= \frac{2(r-1)^2|B_1|+((r-1)^2-1/f(r,k))(|B|-|B_1|)}{|B|}.
	\end{align*}
	A straightforward rearrangement yields
	$$|B_1|\geq \frac{|B|}{2r^2f(r,k)}\stackrel{\eqref{eq:sizeofB2}}{\geq} \frac{n^3}{4\cdot 10^{6}r^4 f(r,k)}\;.$$
	Since, by assumption, no component in $B_1$ has more than $f(r,k)$ vertices, the number of dense components in $B$ must be at least
	$$\frac{|B_1|}{f(r,k)}\geq \frac{n^3}{4\cdot 10^{6}r^4 f(r,k)^2}>r^{-30k^2}n^3\;,
	$$
	and the statement of the lemma follows.
\end{proof}
\section{Dense component exploration}\label{sec:mainlemma}
Our goal in this section is to formulate a strengthening of Lemma~\ref{lem:mainweak} that can be conveniently proved by induction. For this, we need first to define a special class of $((r-2)i+3,i)$-configurations, which can be viewed as natural analogues of $2$-uniform trees. For a hypergraph $\F$ and its subhypergraph $\E$ we write $\F\setminus \E$ to denote the hypergraph with the edge set $E(\F)\setminus E(\E)$, and the vertex set being the union of its edges. When $\E$ consists of a single edge $T$, we write $\F\setminus \{T\}$ for $\F\setminus \E$.

\begin{definition}\label{def:inductive}
	An $((r-2)i+3,i)$-configuration $\F$ is called \emph{inductive} if either $i=2$, or $i>2$ and there exists a hyperedge $T\in E(\F)$ such that:
	\begin{itemize}
		\item $T$ is contained in a ${\cal C}^r_3$,
		\item $T$ has $r-2$ vertices of degree $1$, and
		\item $\F\setminus \{T\}$ is inductive.
	\end{itemize}
\end{definition}
\noindent
The following lemma, which strengthens Lemma~\ref{lem:mainweak}, is the main technical result of the present paper.
\begin{lemma}\label{lem:main}
	Suppose that $r\geq 3$, $\G$ is a linear $r$-graph, and let $B=B(\G)$ be its bowtie graph. Suppose that  $k\geq 3$, and that there exist a vertex $u_0\in V(\G)$, a hyperedge $T_0\in E(\G)$ with $u_0\in T_0$, a set of $2rk$ further hyperedges $\T^0=\{T^0_1,\ldots ,T^0_{2rk}\}\subseteq E(\G)$, and a set of distinct dense components of $B$, $\C=\{C^1,\ldots, C^{2rk}\}$ such that, for each $\ell$, we have  $u_0\in T_\ell^0$ and $\{T_0,T^0_\ell\}\in C^\ell$.
	
	Then for each $2\leq i\leq k$ there exists an $((r-2)i+3,i)$-configuration $\F_i\subset \G$ of one of the following two types:
	\begin{itemize}
		\item [\textbf{Type~1.}] $\F_i$ is an $((r-2)i+2,i)$-configuration with $T_0\in E(\F_i)$.
		\item [\textbf{Type~2.}] There exist a subhypergraph $\E_i\subseteq \F_i$ and a component $C_i\in \C$ such that the following conditions hold.
		\begin{enumerate}
			\item[(P1)] $\E_i$ is an inductive $((r-2)j+3,j)$-configuration for some $j\geq 2$ with $T_0\in E(\E_i)$,
			\item[(P2)] $V(\E_i)\cap V(\F_i\setminus\E_i)\subseteq T_0$,
			\item[(P3)] The set $A_i=\{b\in V(C_i): b=\{T,S\}; T,S\in \E_i\}$ is not empty, and 
			$d_{avg}(B[A_i])<3(r-1)$.
		\end{enumerate}
	\end{itemize}
\end{lemma}

Before giving the proof, let us explain the reasoning underpinning Lemma~\ref{lem:main}. 
Recall that in Lemma \ref{lem:smallcomps} we have seen that if a component $C$ of $B$ contains a long path
then we can find large $((r-2)i+3,i)$-configurations. A first attempt at proving Lemma \ref{lem:main} would then be
to use the fact that we have many components, in order to find many (potentially small) $((r-2)i+3,i)$-configurations
in each of these components, and then somehow merge them into a single large $((r-2)k+3,k)$-configuration. Note however that just taking a disjoint union
of them would not work, since it would not produce a $((r-2)k+3,k)$-configuration. So the idea behind the induction stated above is to use the fact that
these components are {\em dense} in order to devise a way in which these small $((r-2)i+3,i)$-configurations {\em can} be merged into a single
$((r-2)k+3,k)$-configuration. This is done as follows: at a ``typical'' step $i$, the process has a configuration of $\textbf{Type~2}$,
meaning that it is growing an inductive configuration ${\cal E}_i$
within a component $C_i$ (property (P1)). A crucial feature of inductive configurations is that they correspond to subgraphs of $B$ of small average degree (property\footnote{Although the only feature of (P3) we use in the proof is that it implies that $A_i$ is a proper subset of $C_i$, we maintain the stronger
(P3) since it is easier to track in the induction process.} (P3)).
Hence, there must be a vertex $b \in C_i \setminus A_i$ which has not been explored yet, implying that another hyperedge can be added to ${\cal E}_i$.
Now there are two cases. If the new edge creates a new inductive configuration then we still have a $\textbf{Type~2}$ configuration so we can continue growing ${\cal E}_i$ (Case 2.2 in the proof) . Otherwise (Case 2.1 in the proof) we get a denser configuration (i.e. of $\textbf{Type~1}$) so we can pay the cost of moving to a new configuration (with the help of (P2)), and then restart the process (Case 1 in the proof).
%
%

\begin{proof}
	
	We proceed by induction on $i$.
	
	\noindent
	\textbf{Base case $i=2$:}
	
	Set $\F_2:=\{T_0,T^0_1\}$. Then $\F_2$ is of Type~2 with $\E_2:=\F_2$ and $C_2:=C^1$,
	as $\E_2$ is an inductive $(2r-1,2)$-configuration with $T_0\in E(\E_2)$ (so (P1) holds), $V(\F_2\setminus \E_2)=\emptyset$ (so (P2) holds) and $A_2=\{b\}$, where $b=\{T_0,T^0_1\}$ (so (P3) holds).
	
	\noindent
	\textbf{Induction step $i\rightarrow i+1$:}
	
	\noindent
	\textbf{Case 1:} $\F_i$ is of Type~1.
	
	Since $|V(\F_i)|\leq (r-2)i+2\leq rk$ yet $|\T^0|=2rk$, there exists a $T_\ell^{0}\in \T^0$ satisfying $T_\ell^0\cap V(\F_i)=\{u_0\}$. Define $\F_{i+1}:=\F_{i}\cup \{T_{\ell}^0\}$. Then
	$$|V(\F_{i+1})|= |V(\F_{i})\cup T^0_{\ell}|\leq ((r-2)i+2)+r-1=(r-2)(i+1)+3\;,
	$$
	so $\F_{i+1}$ is an $((r-2)(i+1)+3,i+1)$-configuration. We claim $\F_{i+1}$ is of Type~2, with $\E_{i+1}:=\{T_0,T_{\ell}^0\}$ and $C_{i+1}:=C^{\ell}$. First, $T_0\in E(\E_{i+1})$ by definition, and $\E_{i+1}$ is an inductive $(2r-1,2)$-configuration, so (P1) holds. Second, by assumption on $T^0_{\ell}$, we have
	\begin{align*}
	V(\E_{i+1})\cap V(\F_{i+1}\setminus \E_{i+1})&= (T_0\cup T^0_\ell)\cap V(\F_i\setminus \{T_0\})\\
	&= (T_0 \cap V(\F_i\setminus \{T_0\})) \cup (T^0_\ell \cap V(\F_i\setminus \{T_0\}))\\
	&\subseteq T_0 \cup \{u_0\} = T_0,
	\end{align*}
	so (P2) holds. Third, $A_{i+1}=\{b\}$, where $b=\{T_0,T^0_{\ell}\}$, so (P3) holds.
	\medskip
	
	\noindent
	\textbf{Case 2:} $\F_i$ is of Type~2. By the induction hypothesis $\F_i$ satisfies (P1)--(P3) with some $\E_i$, $j$, $C_i$, and $A_i$ as stated therein.
	
	By (P3) we have $A_i\neq\emptyset$, and $d_{avg}(B[A_i])<3(r-1)$. Since $C_i$ is a dense component, $A_i$ must be a proper subset of $V(C_i)$, which means there is a vertex $b\in V(C_i)\setminus A_i$ adjacent to a vertex $b'\in A_i$. By Proposition~\ref{prop:B}(1), we can write $b=\{T,Q_1\}$ and $b'=\{Q_1,Q_2\}$, for some $Q_1,Q_2,T\in E(\G)$, forming a ${\cal C}^r_3$. Note that, since $b'\in A_i$, the definition of $A_i$ implies $Q_1,Q_2\in E(\E_i)$. Consequently, since $b\in V(C_i)\setminus A_i$ and $Q_1\in E(\E_i)$, we must have $T\notin E(\E_i)$. Denote by $w_1$ and $w_2$ the vertices given by $T\cap Q_1=\{w_1\}$ and $T\cap Q_2=\{w_2\}$. 
	
	Suppose that $T\in E(\F_i \setminus \E_i)$. Since $w_1\in Q_1\subseteq V(\E_i)$, and similarly for $w_2$, we would have
	$$\{w_1,w_2\}\subseteq V(\E_i)\cap V(\F_i \setminus \E_i)\stackrel{\text{(P2)}}{\subseteq} T_0\;.$$
	Since, by linearity of $\G$, $w_1$ and $w_2$ can be contained in at most one edge of $\G$, this would imply $T=T_0$, contradicting that $T\notin E(\E_i)$. Hence, $T\notin E(\F_i \setminus \E_i)\cup E(\E_i)=E(\F_i)$. 
	
	Set $\F_{i+1}:=\F_i\cup \{T\}$, that is, $E(\F_{i+1})=E(\F_i)\cup \{T\}$, and $V(\F_{i+1})=V(\F_i)\cup T$. By the above, $|E(\F_{i+1})| =i+1$. We now have two subcases.
	
	\textbf{Case 2.1:} $(T\setminus\{w_1,w_2\})\cap V(\F_i)\neq \emptyset$; equivalently, $|T\cap V(\F_i)|\geq 3$. We claim that in this case $\F_{i+1}$ is of Type~1. By (P1), $T_0\in \E_i \subseteq \F_i\subseteq \F_{i+1}$. To see that $\F_{i+1}$ is an $((r-2)(i+1)+2,i+1)$-configuration, observe that
	$$|V(\F_{i+1})|=|V(\F_i)\cup T|=|V(\F_i)|+|T|-|V(\F_i)\cap T|\leq (r-2)i+3+r-3=(r-2)(i+1)+2\;.$$
 This completes the induction step in this subcase.
	
	\textbf{Case 2.2:} $T \cap V(\F_i)= \{w_1,w_2\}$. 
	In this case
	$$|V(\F_{i+1})|=|V(\F_i)|+|T|-|T\cap V(\F_i)|\leq (r-2)i+3+r-2=(r-2)(i+1)+3\;,$$
	So, $\F_{i+1}$ is an $((r-2)(i+1)+3,i+1)$-configuration.
	We claim that $\F_{i+1}$ is of Type~2 via $\E_{i+1}:=\E_i\cup \{T\}$ and $C_{i+1}:=C_{i}$ (we henceforth omit the subscript, and write $C$).
	
	To verify (P1), note that
	\begin{itemize}
		\item $T$ is contained in $\{Q_1,Q_2,T\}\subseteq E(\E_{i+1})$, which is a ${\cal C}^r_3$,
		\item Each $w \in T\setminus\{w_1,w_2\}$ has degree $1$ in $\E_{i+1}$,
		as by assumption $(T\setminus \{w_1,w_2\})\cap V(\F_i)=\emptyset$, so
		$$d_{\E_{i+1}}(w)\leq d_{\F_{i+1}}(w)=1+d_{\F_{i}}(w)=1\;,
		$$
		and
		\item $\E_i$ is an inductive $((r-2)j+3,j)$-configuration by (P1) in the induction hypothesis.
	\end{itemize}
	Therefore, $\E_{i+1}$ is an inductive $((r-2)(j+1)+3,j+1)$-configuration. Also, by (P1) in the induction hypothesis, $T_0\in E(\E_i)\subseteq E(\E_{i+1})$. Thus (P1) holds.
	
	To verify (P2) note that
	\begin{align*}
	V(\E_{i+1})\cap V(\F_{i+1} \setminus \E_{i+1})&= V(\E_i \cup \{T\})\cap V((\F_i\cup \{T\}) \setminus (\E_i\cup \{T\}))\\
	&=(V(\E_i)\cup (T\setminus\{w_1,w_2\}))\cap V(\F_i\setminus \E_i)\\
	&\subseteq (V(\E_i)\cap V(\F_i\setminus \E_i))\cup ((T\setminus\{w_1,w_2\})\cap V(\F_i))\\
	&=V(\E_i)\cap V(\F_i\setminus \E_i)\subseteq T_0\;,
	\end{align*}
	where the last inclusion is by (P2) in the inductive hypothesis.
	
	It thus remains to verify (P3). To this end, note that since $\E_{i+1}=\E_i\cup \{T\}$, we have
	\begin{align*}
	A_{i+1}&=\{\{S_1,S_2\}\in V(C):S_1, S_2\in E(\E_i)\}\cup \{\{S,T\}\in V(C): S\in E(\E_i)\}\\
	&= A_{i}\cup \{\{S_1,T\}\in V(C): S_1\in E(\E_i), w_1\in S\}\cup \{\{S_2,T\}\in V(C): S_2\in E(\E_i), w_2\in S\}\\
	&=:A_i\cup A'_1 \cup A'_2\;,
	\end{align*}
	In particular, using (P3) in the induction hypothesis, $\emptyset\neq A_i\subseteq A_{i+1}$, so $A_{i+1}$ is not empty.
	Note also that the above union is disjoint, since no bowtie in $A_i$ contains $T$, and no hyperedge $S\in E(\E_i)$ contains both $w_1$ and $w_2$ (as $\{w_1,w_2\}\subset T$). Therefore,
	\begin{equation*}
	|A_{i+1}|=|A_i|+|A'_1|+|A'_2|\;,
	\end{equation*}
	and to complete the proof we need to show that
	\begin{equation}\label{eq:92}
	e(B[A_{i+1}])< \frac{3(r-1)}{2}(|A_i|+|A'_1|+|A'_2|)\;.
	\end{equation}
	Define $M_1=\{t\in\binom{V(\G)}{r-1}:\{t\cup \{w_1\}, T\}\in A'_1\}$ and define $M_2$ analogously. Note that
	\begin{equation}\label{eq:AiMi}
	|M_1|=|A'_1| \text{ and } |M_2|=|A'_2|\;.
	\end{equation}
	We claim that, crucially,
	\begin{equation}\label{eq:crucial}
	e(B[A_i,A'_1])=e([A_i,A'_2])=e(B[A'_1,A'_2])=|V(M_1)\cap V(M_2)|\;.
	\end{equation}
	To see that~\eqref{eq:crucial} would indeed imply~\eqref{eq:92}, observe that each bowtie $b_1=\{S,T\}\in A'_1$ has $S\cap T = \{w_1\}$ as its centre. Therefore, by Proposition~\ref{prop:B}(4), we have $e(B[A'_1])=0$, and, similarly, $e(B[A'_2])=0$. Hence,~\eqref{eq:crucial} would imply
	\begin{equation}\label{eq:indstep}
	e(B[A_{i+1}])- e(B[A_{i}])=e(B[A_i,A'_1,A'_2])=3 |V(M_1)\cap V(M_2)|\;.
	\end{equation}
	\noindent
	Since
	\begin{equation}\label{eq:minvsavg}
	|V(M_1)\cap V(M_2)|\leq \min\{|V(M_1)|,|V(M_2)|\}\leq \frac{|V(M_1)|+|V(M_2)|}{2}=\frac{r-1}{2}(|M_1|+|M_2|)\;,
	\end{equation}
	combining,~\eqref{eq:AiMi},~\eqref{eq:indstep},~\eqref{eq:minvsavg} and (P3) in the induction hypothesis we obtain
	\begin{align*}
	e(B[A_{i+1}])&\stackrel{\eqref{eq:indstep}}{=}e(B[A_i])+3|V(M_1)\cap V(M_2)|\stackrel{\eqref{eq:minvsavg}}{\leq} e(B[A_i])+\frac{3(r-1)}{2}(|M_1|+|M_2|)\\
	&\stackrel{\text{(P3)},\eqref{eq:AiMi}}{<}\frac{3(r-1)}{2}(|A_i|+|A'_1|+|A'_2|)\;.
	\end{align*}
	
	It remains to prove~\eqref{eq:crucial}. To this end, for two hyperedges $S_1, S_2 \in E(\E_i)$ with $w_1\in S_1$ and $w_2\in S_2$, call a ${\cal C}^r_3$, comprising three hyperedges $T,S_1,S_2$ \emph{nice} if, with $b:=\{S_1, S_2\}$, $b_1:=\{S_1, T\}$ and $b_2:=\{S_2, T\}$, we have
	$\{b,b_1,b_2\}\subseteq V(C)$. 
	
	Let $N$ denote the set of all nice configurations. We claim that the tripartite graph $B[A_i,A'_1,A'_2]$ can be partitioned into $|N|$ edge-disjoint triangles. Indeed, given a nice configuration $\{S_1,S_2,T\}$ we claim that
	\begin{itemize}
		\item[(1)] $b$, $b_1$ and $b_2$ belong to $A_i$, $A'_1$ and $A'_2$, respectively, and form a triangle in $B$,
		\item[(2)] Two nice configurations define two edge-disjoint triangles as above, and
		\item[(3)] Every edge of $B[A_i,A'_1,A'_2]$ belongs to one of the above triangles.
	\end{itemize}
	Property (1) is immediate from the definitions. To see (2) and (3), consider an edge in $B[A_i,A'_1,A'_2]$, e.g. between $b=\{S_1,S_2\}\in A_i$ and $b_1=\{S_1,T\}\in A'_1$ (the other options can be handled similarly). By definition of $A_i$ we have $S_1,S_2\in E(\E_i)$ and $b\in V(C)$. Moreover, since $b_1\in A'_1$ we have $w_1\in S_1$, and since $\{b,b_1\}\in E(B)$, Proposition~\ref{prop:B}(1) implies that $\{S_1,S_2,T\}$ is a ${\cal C}^r_3$,. This means $S_2\cap T= (V(\E_i)\cap T)\setminus \{w_1\}=\{w_2\}$, thus $w_2\in S_2$. Let $b_2:=\{S_2,T\}$, and note that $b,b_1$ and $b_2$ form a triangle in $B$, which, as $b\in V(C)$, implies $\{b,b_1,b_2\}\subseteq V(C)$. Hence, $\{S_1,S_2,T\}$ is a nice configuration. Furthermore, the correspondence between $\{S_1,S_2,T\}$ and $(b,b_1)=(\{S_1,S_2\}, \{S_1,T\})$ is unique.
	
	%
	%
We now claim that $|N|=|V(M_1)\cap V(M_2)|$. To see this, note that, by linearity of $\G$, each of $M_1$ and $M_2$ is a collection of disjoint $(r-1)$-sets in $V(\G)$, and each $t_1\in M_1$ and $t_2\in M_2$ intersect in at most one vertex.
Thus, each vertex $v =t_1\cap t_2$, with $(t_1,t_2)\in M_1\times M_2$ belongs to unique $t_1$ and $t_2$, and this defines uniquely a ${\cal C}^r_3$, consisting of the hyperedges $T, S_1, S_2$, where $S_1=t_1\cup \{w_1\}$ and $S_2=t_2\cup \{w_2\}$, as $S_1\cap S_2=\{v\}, S_1\cap T=\{w_1\}, S_2\cap T=\{w_2\}$, and $v,w_1,w_2$ are distinct. By definition of $M_1$  we have $b_1:=\{S_1,T\}\in A'_1$, so $S_1\in E(\E_i)$ and $b_1\in V(C)$, and, by the same reasoning, $S_2\in E(\E_i)$ and $b_2:=\{S_2,T\}\in V(C)$. Since $b_1$ and $b:=\{S_1,S_2\}$ are adjacent in $B$, we also have $b\in V(C)$.
Hence, $\{T,S_1, S_2\}$ is a nice configuration. Conversely, every nice configuration $\{T,S_1, S_2\}$ defines a vertex $v\in V(\G)$ by $\{v\}= S_1\cap S_2$, and it is easy to check that $v\in V(M_1)\cap V(M_2)$, and that the two mappings are inverse bijections.
	
The statements in the above two paragraphs establish~\eqref{eq:crucial} and thus complete the proof.
\end{proof}

\section*{Acknowledgement}

The authors would like to thank Rajko Nenadov and Benny Sudakov for helpful early discussions, and Lior Gishboliner for technical help producing this manuscript. We further thank an anonymous referee for their helpful comments.

\end{document}